\documentclass[12pt,a4paper]{article}
\usepackage{amsmath,amssymb,amsthm}
\usepackage{amssymb,latexsym, bbm,comment}
\usepackage{graphicx}
\usepackage{makeidx}
\newtheorem{theorem}{Theorem}[section]
\newtheorem{prop}[theorem]{Proposition}
\numberwithin{equation}{section}

\newcommand{\R}{\mathbb{R}}
\newcommand{\Rd}{\mathbb{R}^d}
\newcommand{\Rm}{\mathbb{R}^m}

\newcommand{\C}{\mathbb{C}}
\newcommand{\E}{\mathbb{E}}

\newcommand{\ds}{\displaystyle}
\newcommand{\g}{\mathfrak{g}}

\newcommand{\p}{\mathfrak{p}}
\newcommand{\fk}{\mathfrak{k}}

\newcommand{\bean}{\begin{eqnarray*}}
\newcommand{\eean}{\end{eqnarray*}}

\newcommand{\la}{\langle}
\newcommand{\ra}{\rangle}

\newcommand{\Z}{\mathbb{Z}}

\newcommand{\G}{\widehat{G}}

\newcommand{\Ad}{\rm{Ad}}

\newcommand{\limn}{\lim_{n \rightarrow \infty}}

\begin{document}

\date{}

\title{The Positive Maximum Principle on Symmetric Spaces}

\author{David Applebaum, Trang Le Ngan\\ School of Mathematics and Statistics,\\ University of
Sheffield,\\ Hicks Building, Hounsfield Road,\\ Sheffield,
England, S3 7RH\\ ~~~~~~~\\e-mail: D.Applebaum@sheffield.ac.uk, tlengan1@sheffield.ac.uk}

\maketitle

\begin{abstract} We investigate the Courr\`{e}ge theorem in the context of linear operators that satisfy the positive maximum principle on a space of continuous functions over a symmetric space. Applications are given to Feller--Markov processes. We also introduce Gangolli operators, which satisfy the positive maximum principle, and generalise the form associated with the generator of a L\'{e}vy process on a symmetric space. When the space is compact, we show that Gangolli operators are pseudo--differential operators having scalar symbols.
\end{abstract}

\section{Introduction}

Consider a linear operator $A$ defined on the space $C_{c}^{\infty}(\Rd)$ of smooth functions of compact support. If it satisfies the positive maximum principle (PMP), then by a classical theorem of Courr\`{e}ge \cite{Courr} $A$ has a canonical form as the sum of a second--order elliptic differential operator and a non--local integral operator. Furthermore $A$ may also be written as a pseudo--differential operator whose symbol is of L\'{e}vy--Khintchine type (but with variable coefficients). This result is of particular importance for the study of Feller--Markov processes in stochastic analysis. The infinitesimal generator of such a process always satisfies the PMP, and so has the canonical form just indicated. The use of the symbol to investigate probabilistic properties of the process has been an important theme of much recent work in this area (see e.g. \cite{Jac4}, \cite{BSW} and references therein).

In a recent paper \cite{AT1}, the authors generalised the Courr\`{e}ge theorem to linear operators satisfying the PMP on a Lie group $G$. The key step was to replace the set of first order partial derivatives $\{\partial_{1}, \ldots, \partial_{d}\}$ by a basis $\{X_{1}, \ldots, X_{d}\}$ for the Lie algebra $\g$ of $G$. In this case, when $G$ is compact, we find that the operator is a pseudo--differential operator in the sense of Ruzhansky and Turunen \cite{RT}, with matrix--valued symbols, obtained using Peter--Weyl theory.

In the current paper, we extend the Courr\`{e}ge theorem to symmetric spaces $M$. Since any such space is a homogeneous space $G/K$, where $K$ is a compact subgroup of the Lie group $G$, we may conjecture that the required result can be obtained from that of \cite{AT1} by use of projection techniques; however, this is not the case as a linear operator that satisfies the PMP on $C_{c}^{\infty}(G/K)$ may not satisfy it on $C_{c}^{\infty}(G)$. Nonetheless, a straightforward variation on the proof given in \cite{AT1} does enable us to derive the required result. As we might expect, the symmetries that are brought into play by the subgroup $K$, imposes constraints on the coefficients of the operator $A$.

The key probabilistic application of our result is that we obtain a canonical form for the generators of Feller processes on symmetric spaces. The case where the transition probabilities of $M$--valued Markov processes are $G$--invariant is the topic of a recent research monograph \cite{LiaoN}. There has been considerable interest in the case where such a process arises as the projection of a $K$--bi--invariant L\'{e}vy process in $G$ as in this case there is an analogue of the L\'{e}vy--Khintchine formula (due to Gangolli \cite{Gang1}). For recent work in this area, see e.g. \cite{AT0} and references therein.

When we come to study pseudo--differential operators, we again assume that $G$ (and hence $M$) is compact. We study a class of linear operators, which we call {\it Gangolli operators} (in recognition of the important contributions of Ramesh Gangolli \cite{Gang1} to this field). These satisfy the PMP, and their structure generalises those obtained from the $K$--bi--invariant L\'{e}vy processes. They all have a scalar--valued symbol which is defined using the spherical transform, rather than the full Fourier transform on the group as in \cite{App5}.

Many of the results of this paper first appeared in the PhD thesis \cite{Lili}, but our approach here is a little different.

\vspace{5pt}

{\bf Notation.} Throughout this paper, $G$ is a Lie group having neutral element $e$, dimension $d$ and Lie algebra $\g$. The exponential map from $\g$ to $G$ will be denoted as $\exp$. The Borel $\sigma$--algebra of $G$ is denoted as ${\mathcal B}(G)$. We denote by ${\mathcal F}(G)$, the linear space of all real--valued functions on $G$, $B_{b}(G)$ the Banach space (with respect to the supremum norm $||\cdot||_{\infty}$) of all bounded Borel measurable real--valued functions on $G$,  $C_{0}(G)$ the closed subspace of all continuous functions on $G$ that vanish at infinity, and $C_{c}^{\infty}(G)$ the dense linear manifold in $C_{0}(G)$ of smooth functions with compact support.

\section{Preliminaries on Lie Groups and Symmetric Spaces}

Let $M$ be a globally Riemannian symmetric space. Then (see Theorem 3.3 in \cite{Hel1} p.208)  there exists a Lie group $G$ and a compact subgroup $K$ of $G$ such that $M$ is diffeomorphic to the homogeneous space of left cosets $G/K$. As is standard procedure we will identify $M$ with $G/K$ henceforth, and write $\natural$ for the canonical continuous surjection from $G$ to $M$ which maps each $g \in G$ to the coset $gK$. We write $o = \natural(e)$. We have a natural identification between the space $C_{0}(M)$ and the closed subspace $C_{0}(G/K)$ of $C_{0}(G)$ comprising functions on $G$ that are right $K$--invariant.

At the Lie algebra level, we have the Cartan decomposition $\g = \fk \oplus \p$, where $\fk$ is the Lie algebra of $K$ and $d\natural_{e}$ is a linear isomorphism between $\p$ and $T_{o}(M)$ (see e.g. Theorem 3.3 in \cite{Hel1} pp.208--9). There is an Ad$(K)$--invariant inner product on $\g$ which corresponds to the Riemannian metric on $M$.  We will choose once and for all a basis $\{X_{1}, \ldots , X_{d}\}$ for $\g$ where $d =$ dim$(G)$, and order this so that $\{X_{1}, \ldots, X_{m}\}$ is a basis for $\p$, where $m =$ dim$(M)$.

We fix a system of canonical co--ordinates $(x_{1}, \ldots, x_{d})$  at $e$ which we extend to $G$ such that $x_{i} \in C_{c}^{\infty}(G)$ for $i = 1, \ldots, d$. Following \cite{LW} and \cite{LiaoN} pp. 77--8, we assume that for each $j = 1, \ldots , m, g \in G, k \in K, x_{j}(gk) = x_{j}(g)$ and

\begin{equation} \label{coAd}
\sum_{i=1}^{m}x_{i}(kg)X_{i} = \sum_{i=1}^{m}x_{i}(g)\Ad(k)X_{i}.
\end{equation}

Standard Lie group calculations establish the following for each $f \in C^{\infty}(G/K), g \in G, k \in K, X, Y \in \p$,
\begin{eqnarray} \label{Adcalc}
Xf(gk) & = & (\Ad(k)X)f(g) \nonumber \\
XYf(gk) & = & (\Ad(k)X)(\Ad(k)(Y)f(g),
\end{eqnarray}
and
\begin{eqnarray} \label{Adcalc1}
X(f \circ c_{k})(g) &  = & (\Ad(k)X)f(c_{k}(g)),\nonumber \\
XY(f \circ c_{k})(g) & = & (\Ad(k)X)(\Ad(k)Y)f(c_{k}(g)),
\end{eqnarray}
where $c_{k}(g) = kgk^{-1}$.

We also have that for all $X \in \fk, f \in C^{\infty}(G/K), g \in G$,
\begin{equation} \label{Kzero}
Xf(g) = 0.
\end{equation}

As for each $k \in K, \Ad(k)$ maps $\p$ linearly to $\p$, we can associate to it the $m \times m$ matrix $[\Ad(k)]$, with respect to the basis $X_{1}, \ldots, X_{m}$, in the usual way. We say that a vector $b \in \R^{m}$ is Ad$(K)$--invariant if $b = \Ad(k)^{T}b$ for all $k \in K$. Similarly, an $m \times m$ real--valued matrix $C$ is Ad$(K)$--invariant if $C = \Ad(k)^{T}C\Ad(k)$ for all $k \in K$. If we require that $\{X_{1}, \ldots, X_{m}\}$ is an orthonormal basis for $\p$, then the matrix $[\Ad(k)]$ is unitary, and we may drop the transposes in the definition of Ad$(K)$--invariance for both vectors and matrices.

Let $\sigma: G \rightarrow \mbox{Diff}(M)$ be the left action given by $\sigma(g)hK = ghK$ for all $g, h \in G$. For each $g \in G$, the mapping $\phi_{g}:= d\sigma(g) \circ d\natural_{e}$ is a linear isomorphism between $\p$ and $T_{gK}(M)$. We obtain a useful family of smooth vector fields on $M$ by defining $$\widetilde{X}(gK): = \phi_{g}(X),$$
for each $g \in G, X \in \p$.

\section{The Positive Maximum Principle and the Courr\`{e}ge Theorem}

Let $E$ be a locally compact Hausdorff space and ${\mathcal C}$ be a closed subspace of $C_{0}(E)$, which is the real Banach space of all real--valued continuous functions defined on $E$, equipped with the uniform topology. We also choose ${\mathcal F}$ to be a sub--algebra of the real algebra ${\mathcal F}(E, \R)$ of all real--valued functions defined on $E$.

A linear operator $A: D_{A} \subseteq {\mathcal C} \rightarrow {\mathcal F})$ (where $D_{A}:= $Dom$(A)$) is said to satisfy the {\it positive maximum principle} on ${\mathcal C}$ (PMP for short) if $f \in D_{A}$ and $f(x_{0}) = \sup_{x \in E}f(x) \geq 0$ implies that $Af(x_{0}) \leq 0$. In our previous paper \cite{AT1}, we studied  the PMP with $E = C_{0}(G)$ and ${\mathcal F} = {\mathcal F}(G, \R)$ . In this paper we will take $E$ to be $C_{0}(M)$, or equivalently $C_{0}(G/K)$. We will be interested in a class of distributions on $M$ that we define as follows: A {\it $\p$-induced distribution} $P$ on $M$ is a real--valued linear functional defined on $C_{c}^{\infty}(M)$ such that for every compact set $H$ contained in $M$, there exists $k \in \Z_{+}, C > 0$ so that for all $f \in C_{c}^{\infty}(H)$,
\begin{equation} \label{dist1}
|Pf| \leq C \sum_{|\alpha| \leq k}||\widetilde{X}^{\alpha}f||_{\infty},
\end{equation}
where, as in \cite{AT1} the sum on the right hand side of (\ref{dist1}) is a convenient shorthand for
$$||f||_{\infty} + \sum_{i=1}^m ||\widetilde{X}_i f||_{\infty} + \dots + \sum_{i_1, i_2, \dots, i_k = 1}^m ||\widetilde{X}_{i_1}\widetilde{X}_{i_2} \cdots \widetilde{X}_{i_k} f||_{\infty}.$$
We say that $P$ is of {\it order $k$} if the same $k$ may be used in (\ref{dist1}) for all compact $H \subseteq M$. The set of all $\p$-induced distributions on $M$ is in one--to--one correspondence with a class of distributions on $C^{\infty}(G/K)$ which are defined exactly as above but with each $\widetilde{X}$ replaced by $X$. We will call these $\p$-induced distributions on $G$.  These are clearly very closely related to the class of distributions on $G$ studied in \cite{AT1}.

We say that a linear functional $T:C_{c}^{\infty}(G/K) \rightarrow \R$ satisfies the positive maximum principle (PMP) if $f \in C_{c}^{\infty}(G/K)$ with $f(e) = \sup_{g \in G}f(g) \geq 0$ then $Tf \leq 0$. Similarly to Theorem 3.3 in \cite{AT1}, but also making use of (\ref{Kzero}), we can show that any such linear functional satisfying the PMP is a $p$--induced distribution of order $2$.

In this paper a $K$--right--invariant L\'{e}vy measure $\mu$ on $G$ will be a $K$--right--invariant Borel measure $\mu$ such that

$$ \mu(\{e\}) = 0, \nu(U^{c}) < \infty~\mbox{and}~\int_{U}\sum_{i=1}^{m}x_{i}(g)^{2}\mu(dg) < \infty,$$
for every canonical co--ordinate neighbourhood $U$ of $e$. Such measures are in one--to--one correspondence with L\'{e}vy measures on $M$ as defined in \cite{LW} and \cite{Liao} p.42.

\begin{theorem} \label{PMPdist} Let $T:C_c^\infty(G/K)\to \R$ be a linear functional satisfying the positive maximum principle. Then there exists $c \geq 0, b = (b_{1}, \ldots, b_{m}) \in \Rm$, a non--negative definite symmetric $m \times m$ real--valued matrix $a = (a_{ij})$, and a $K$--right--invariant L\'{e}vy measure $\mu$ on $G$ such that
	
	\begin{eqnarray} \label{distrform}
	Tf & = & \sum_{i,j =1}^m a_{ij}X_iX_j f(e) + \sum_{i=1}^m b_i X_i f(e) - cf(e) \nonumber \\ & + & \int_{G} \left(f(g)-f(e) - \sum_{i=1}^m x_i(g) X_if(e)\right) \mu(dg),
	\end{eqnarray}
	for all $f \in C_c^\infty(G/K)$.
Conversely any linear functional from $C_c^\infty(G/K)$ to $\R$ which takes the form (\ref{distrform}) satisfies the positive maximum principle.
\end{theorem}

\begin{proof} This is established along similar lines to Theorems 3.4 and 3.5 in \cite{AT1}. Note that the role of the function $\sum_{i=1}^{d}x_{i}^{2}$ in that paper is now played by $\sum_{i=1}^{m}x_{i}^{2}$.  Also we exploit the correspondence between $\p$--induced distributions on $M$ and on $G$ to first associate $\mu$ to a $\p$--induced distribution of order zero on $M$ using the Riesz representation theorem in $C_{c}^{\infty}(M)$, and then identify it with a $K$--right--invariant measure on $G$.
\end{proof}

By a $K$--right--invariant L\'{e}vy kernel, we will mean a mapping $\mu: G \times{\cal B}(G) \rightarrow [0, \infty]$ which is such that for each $g \in G, \mu(g, \cdot)$ is a $K$--right--invariant L\'{e}vy measure.

\begin{theorem} \label{PMP1}

\begin{enumerate} \item The mapping $A: C_{c}^{\infty}(G/K) \rightarrow {\mathcal F}(G)$ satisfies the PMP if and only if there exist functions $c, b_{i}, a_{jk}~(1 \leq i,j,k \leq m)$ from $G$ to $\R$, wherein $c$ is non--negative, and the matrix $a(\sigma): = (a_{jk}(\sigma))$ is non--negative definite and symmetric for all $\sigma \in G$, and a $K$--right--invariant L\'{e}vy kernel $\mu$, such that for all $f \in C_{c}^{\infty}(G/K), \sigma \in G$,

\begin{eqnarray} \label{PMP2}
Af(\sigma) & = & -c(\sigma)f(\sigma) + \sum_{i=1}^{m}b_{i}(\sigma)X_{i}f(\sigma) + \sum_{j,k = 1}^{m}a_{jk}(\sigma)X_{j}X_{k}f(\sigma) \nonumber \\
& + & \int_{G}\left(f(\sigma \tau) - f(\sigma) - \sum_{i=1}^{m}x_{i}(\tau)X_{i}f(\sigma)\right)\mu(\sigma, d\tau).
\end{eqnarray}

\item The mapping $A: C_{c}^{\infty}(G/K) \rightarrow {\mathcal F}(G/K)$ satisfies the PMP if and only if there exist $c, b, a , \mu$ as in (1) such that (\ref{PMP2}) holds, and we also have that for all $g \in G, k \in K$,

    \begin{enumerate} \item[(I)] $c(gk) = c(g),$
    \item[(II)] $b(g) = [Ad(k)]^{T}b(gk)$,
     \item[(III)] $ a(g) = [Ad(k)]^{T}a(gk)[Ad(k)],$
     \item[(IV)] $\mu(gk, B) = \mu(g, kB)$, for all $B \in {\mathcal B}(G)$.
     \end{enumerate}

\item The mapping $A: C_{c}^{\infty}(K \backslash G/K) \rightarrow {\mathcal F}(K \backslash G/K)$ satisfies the PMP if and only if there exist $c, b, a , \mu$ as in (1) such that (\ref{PMP2}) holds,  and we also have that for all $g \in G, k, k^{\prime} \in K$,

    \begin{enumerate} \item[(V)] $c(kgk^{\prime}) = c(g),$
    \item[(VI)] $b(g) = b(kgk^{\prime})$ and $b(g)$ is Ad$(K)$--invariant,
     \item[(VII)] $a(g) = a(kgk^{\prime})$ and $a(g)$ is Ad$(K)$--invariant,
     \item[(VIII)] $\mu(gk^{\prime}, B) = \mu(kg, k^{\prime}B)$, for all $B \in {\mathcal B}(G)$.
     \end{enumerate}

     \end{enumerate}

\end{theorem}

\begin{proof} \begin{enumerate}

\item This is proved by defining the linear functional $Af(e)$ and using the result of Theorem \ref{PMPdist} together with left translation, just as in the proof of Theorem 3.6 in \cite{AT1}.

\item This follows from (1) using the fact that we now have $R_{k}Af = Af$ for all $k \in K$, and then applying (\ref{coAd}) and (\ref{Adcalc}) and using uniqueness of $c, b, a$ and $\mu$.

\item We first establish the analogue of Theorem \ref{PMPdist} for real linear functionals defined on $C^{\infty}_{c}(K \backslash G /K)$. The proof is as before, but when we use the Riesz representation theorem, we must identify $C^{\infty}_{c}(K \backslash G /K)$ with $C_{c}^{\infty}(N)$, where $N:=\{KgK; g \in G\}$ is the double--coset space, which is known (and easily shown) to be locally compact and Hausdorff. (V) and (VIII) follow from (1) and (2) and uniqueness of $c, b, a$ and $\mu$, using the fact that in addition to $R_{k^{\prime}}Af = Af$ for all $k^{\prime} \in K$ we also have $L_{k}Af = Af$ for all $k \in K$.
The latter, when combined with (II) and (III) also yield $b(g) = [Ad(k^{\prime})]^{T}b(kgk^{\prime})$ and $a(g) = [Ad(k^{\prime})]^{T}a(kgk^{\prime})[Ad(k^{\prime})],$ for all  $g \in G, k, k^{\prime} \in K$. However since $f \circ c_{k} = f$, for all $k \in K$, we may apply (\ref{Adcalc1}) at $g = e$ to obtain $b(e) = [Ad(k)]^{T}b(e)$ and $a(e) = [Ad(k)]^{T}a(e)[Ad(k)]$. By the construction of Theorem 3.6 in \cite{AT1}, we then find that $b(g) = [Ad(k)]^{T}b(g)$ and $a(g) = [Ad(k)]^{T}a(g)[Ad(k)]$ for all $g \in G$. Combining these with the identities obtained earlier in the proof yield the $K$--bi--invariance in (VI) and (VII).

\end{enumerate}

\end{proof}

Remarks \begin{enumerate} \item In each of the three cases dealt with in Theorem \ref{PMP1}, we can ensure that the range of $A$ is in an appropriate space of continuous functions that vanish at infinity by imposing the conditions of Theorems 3.7 and 3.8 in \cite{AT1}.

\item We may also reformulate each part of Theorem \ref{PMP1} directly in $C_{0}(M)$ by replacing $X_{i}$ with $\widetilde{X}_{i}, i = 1, \ldots, m$. The role of Ad is then played by the isotropy representation of $K$ in $T_{o}(M)$, and in (3) we must introduce the space of functions on $M$ that are $\sigma(K)$--invariant.

\end{enumerate}

Now consider a diffusion operator $B: C_{c}^{\infty}(K \backslash G/K) \rightarrow {\mathcal F}(K \backslash G/K)$ which satisfies the positive maximum principle and takes the form

\begin{equation} \label{diff}
Bf(g) = \sum_{i=1}^{m}b_{i}(\sigma)X_{i}f(g) + \sum_{j,k = 1}^{m}a_{jk}(g)X_{j}X_{k}f(g),
\end{equation}
for each $f \in C_{c}^{\infty}(K \backslash G/K), g \in G$. Then conditions (VI) and (VII) of Theorem \ref{PMP1} (3) hold.

We say that the {\it standard irreducibility conditions} hold for the pair $(K, \p)$ if
\begin{enumerate}
\item $\{X_{1}, \ldots, X_{m}\}$ is an orthonormal basis for $\p$.
\item  $\Ad(K)$ acts irreducibly on $\p$.
\item dim$(M) > 1$.
\end{enumerate}

The following result is well--known when the vector--valued function $b$ and matrix--valued function $a$ is constant (see e.g. Proposition 3.2 in \cite{LiaoN}, p.77).

\begin{theorem} \label{reduce}
Let $B: C_{c}^{\infty}(K \backslash G/K) \rightarrow {\mathcal F}(K \backslash G/K)$ be a diffusion operator of the form (\ref{diff}) and assume that the standard irreducibility conditions hold. Then for each $f \in C_{c}^{\infty}(K \backslash G/K), g \in G$.
$$ Bf(g) = \alpha(g) \Delta f(g),$$
where $\alpha$ is a $K$--bi--invariant mapping from $G$ to $[0, \infty)$ and $\Delta = \sum_{i=1}^{m}X_{i}^{2}$ is the ``horizontal Laplacian''.
\end{theorem}

\begin{proof} (Sketch) For each $k \in K$, we have
$$ [\Ad(k)] = U\Ad(k)U^{-1},$$
where $U$ is the unitary isomorphism from $\p$ to $\R^{m}$ which maps $X_{i}$ to $e_{i}$ for $i = 1, \ldots m$, where $\{e_{1}, \ldots, e_{m}\}$ is the natural basis for $\R^{m}$. It is straightforward to verify that $k \rightarrow [\Ad(k)]$ is an irreducible unitary representation of $K$ on $\R^{m}$. But from (VII) we have that for each $g \in G, k \in K$,
$$ [\Ad(k)]a(g) = a(g)[\Ad(k)],$$
and so $a(g) = \alpha(g)I_{d}$, by Schur's lemma. Since $b(g) \in \p$ for all $g \in G$, then $b(g) =0$ by (VI) and the irreducibility conditions.

\end{proof}

\section{Applications to Feller Processes}

Let $(\Omega, {\mathcal F}, P)$ be a probability space wherein ${\mathcal F}$ is equipped with a filtration of sub--$\sigma$--algebras. Let $X = (X(t), t \geq 0)$ be a (homogeneous) Markov process defined on $\Omega$ and taking values in $G$. We define transition probabilities in the usual way so for each $t \geq 0, \sigma \in G, B \in {\mathcal B}(G)$,
$$ p_{t}(\sigma, B) = P(X(t) \in B|X(0) = \sigma).$$
We say that the process $X$ is {\it $K$--right--invariant} if
$$  p_{t}(\sigma k, Bk^{\prime}) = p_{t}(\sigma, B)$$
for all $t \geq 0, \sigma \in G, B \in {\mathcal B}(G), k, k^{\prime} \in K$. Define the transition operators $(T_{t}, t \geq 0)$ for the process $X$ by the prescription
\begin{equation} \label{TO}
T_{t}f(\sigma) = \E(f(X_{t})|X(0) = \sigma) = \int_{G}f(\tau)p_{t}(\sigma,d\tau),
\end{equation}
for all $f \in B_{b}(G), \sigma \in G, t \geq 0$. If $X$ is $K$--right--invariant, it follows from (\ref{TO}) that $T_{t}$ preserves the space $B_{b}(G/K)$ for all $t \geq 0$. We say that $X$ is a {\it $K$--right--invariant Feller process} if $(T_{t}, t \geq 0)$ is a $C_{0}$--semigroup\footnote{Note that we do not require $(T_{t}, t \geq 0)$ to be a $C_{0}$--semigroup on $C_{0}(G)$, and so $X$ may not be a Feller process on $G$, in the usual sense.} on the Banach space $C_{0}(G/K)$. If $A$ denotes the infinitesimal generator of $(T_{t}, t \geq 0)$, then a standard argument using (\ref{TO}) shows that $A$ satisfies the positive maximum principle (see e.g. Lemma 4.1 in \cite{AT1}). Hence if $C_{c}^{\infty}(G/K) \subseteq$ Dom$(A)$, then $A$ takes the form (\ref{PMP2})on $C_{c}^{\infty}(G/K)$ with the conditions of Theorem \ref{PMP1}(2) holding.

If $X$ is a $K$--right--invariant Markov process on $G$, the prescription $Y(t) = \natural(X(t))$ induces a Markov process $Y = (Y(t), t \geq 0)$ on $M$, and then the transition probabilities of $Y$ are given by
$$ q_{t}(\sigma K, BK) = p_{t}(\sigma, B),$$
for all $t \geq 0, \sigma \in G, B \in {\mathcal B}(G)$, where $BK:=\{gk, g \in B, k \in K\}$. The transition operators of $Y$ are defined by
$$ S_{t}f(x) = \int_{M}f(y)q_{t}(x, dy),$$
 and it is easy to see that
$$ T_{t}(f \circ \natural) = (S_{t}f) \circ \natural,$$
for all for all $f \in B_{b}(M), x \in M, t \geq 0$ (c.f. Proposition 1.16 in \cite{LiaoN}, p.16). If $X$ is a $K$--right--invariant Feller process in $G$, then $Y$ is a Feller process in $X$.

We say that a $K$--right--invariant Markov process in $G$ is {\it spherical} if its transition probabilities are also left $G$--invariant, i.e.
$$  p_{t}(g\sigma k, gBk^{\prime}) = p_{t}(\sigma, B)$$
for all $t \geq 0, g,\sigma \in G, B \in {\mathcal B}(G), k, k^{\prime} \in K$. It is not difficult to check that if $X$ is spherical Markov, then $T_{t}$ preserves the space $B_{b}(K \backslash G/K)$ for all $t \geq 0$.  In this case $Y$ is a $G$--invariant Markov process in $M$, as discussed in section 1.1 of \cite{LiaoN}. We say that a spherical Markov process is {\it spherical Feller} if $(T_{t}, t \geq 0)$ is a $C_{0}$--semigroup on the Banach space $C_{0}(K \backslash G/K)$. In this case, if $C_{c}^{\infty}(K \backslash G/K) \subseteq$ Dom$(A)$, then $A$ takes the form (\ref{PMP2}) on $C_{c}^{\infty}(K \backslash G/K)$ with the conditions of Theorem \ref{PMP1}(3) holding.

An important example of a spherical Feller process is a {\it spherical L\'{e}vy process}. This is essentially a (left) L\'{e}vy process (i.e. a process with stationary and independent increments) in $G$ having $K$--bi--invariant laws \cite{Liao}, with $X(0)$ being uniformly distributed on $K$. Define $\rho_{t}(B) = P(X(t) \in B)$ for all $t \geq 0, B \in {\mathcal B}(G)$. Then $(\rho_{t}, t \geq 0)$ is a weakly--continuous convolution semigroup of $K$--bi--invariant probability measures on $G$, wherein $\rho_{0}$ is normalised Haar measure on $K$. In this case the transition probabilities are given by $p_{t}(\sigma B) = \rho_{t}(\sigma^{-1}B)$ for each $t \geq 0, \sigma \in G,  B \in {\mathcal B}(G)$. An interesting case is obtained by assuming that the standard irreducibility conditions hold. It then follows by Theorems \ref{PMP2} (3) and \ref{reduce} (see also arguments in \cite{App1}, \cite{LW} and section 5.5 in \cite{LiaoN}) that the generator

\begin{equation} \label{Hunt}
Af(g) = a \Delta f(g) + \int_{G}\left(f(\sigma \tau) - f(\sigma) - \sum_{i=1}^{m}x_{i}(\tau)X_{i}f(\sigma)\right)\mu(d\tau)
\end{equation}

for each $f \in C_{c}^{\infty}(K \backslash G / K), g \in G$ where $a \geq 0$ and $\mu$ is a $K$--bi--invariant L\'{e}vy measure in $G$. 
Note that (\ref{Hunt}) is a special case of (\ref{PMP2}) wherein $b = c = 0$, the matrix--valued function $a$ is a constant multiple of the identity, and the L\'{e}vy kernel reduces to a L\'{e}vy measure. We refer readers to  Theorem 4.5 in \cite{AT1} for further discussion of the relation between translation invariance of the semigroup and the appearance of constant coefficients in Theorem \ref{PMP1}. We remark that the operator $\Delta$ is the lift of the Laplace--Beltrami operator $\Delta_{M}$ in $M$, in that $\Delta (f \circ \natural) = \Delta_{M}f \circ \natural$, for all $f \in C_{c}^{\infty}(M)$. Under these conditions we also have an extension of Gangolli's L\'{e}vy--Khinchine formula \cite{Gang1}, due to Liao and Wang \cite{LW}, Theorem 2 (see also Theorem 5.3 in \cite{LiaoN}, p.139). To be precise, if $\phi$ is a bounded spherical function\footnote{See e.g. \cite{Hel2} for background on spherical functions.} on $G$ then for all $t \geq 0$
\begin{equation} \label{GLK1}
\E(\phi(X(t)) = \int_{G}\phi(g)\rho_{t}(dg) = e^{-t \eta_{\phi}},
\end{equation}
where
\begin{equation} \label{GLK2}
\eta_{\phi} = c_{\phi} + \int_{G}(1 - \phi(g))\mu(dg),
\end{equation}
and $c_{\phi} \geq 0$ with $\Delta \phi = c_{\phi} \phi.$

\section{A Class of Pseudo--differential Operators in Compact Symmetric Spaces}

We begin by considering the following class of linear operators that satisfy the positive maximum principle:

\begin{eqnarray} \label{PMP5}
Af(\sigma) & = &  a(\sigma) \Delta f(\sigma) \nonumber \\
& + & \int_{G}\left(f(\sigma \tau) - f(\sigma) - \sum_{i=1}^{m}x_{i}(\tau)X_{i}f(\sigma)\right)\mu(\sigma, d\tau),
\end{eqnarray}
for all $f \in C_{c}^{\infty}(K \backslash G / K), \sigma \in G$.

Note that here $a$ is a scalar--valued function and that, if the irreducibility conditions hold, then the form of the second order part of the generator is determined by Theorem \ref{reduce}. Clearly (\ref{Hunt}) is a special case of (\ref{PMP5}), wherein the coefficients are constant.

We assume that the conditions of Theorems 3.7 and 3.8 in \cite{AT1} hold as well as those of Theorem \ref{PMP1}(3), so that $A$ maps $C_{c}^{\infty}(K \backslash G / K)$ to $C_{0}(K \backslash G / K)$. Note that in particular, the function $a$  is now required to be continuous. We further impose the first moment condition
$$ \sup_{\sigma \in G}\int_{G}||x(\tau)||_{\R^{m}}\mu(\sigma, d \tau) < \infty.$$ Then by Theorem \ref{reduce}, we have
$$\int_{G}x_{i}(\tau)\mu(\sigma, d\tau) = 0, $$
for all $i=1, \ldots, m, \sigma \in G$, and we may rewrite (\ref{PMP5}) as

\begin{equation} \label{PMP6}
Af(\sigma) = a(\sigma) \Delta f(\sigma) + \int_{G}(f(\sigma \tau) - f(\sigma))\mu(\sigma, d\tau),
\end{equation}

We call operators of the form (\ref{PMP6}) {\it Gangolli operators}, for reasons that will become clearer below.

If $\phi$ is a bounded spherical function on $G$ we have for all $g \in G$

\begin{equation} \label{eig}
A \phi(g) = - \eta_{g, \phi}\phi(g)
\end{equation}

where

\begin{equation} \label{eig1}
\eta_{g, \phi} = a(g)c_{\phi} + \int_{G}(1 - \phi(\tau))\mu(g, d\tau)
\end{equation}

and we see that (\ref{eig1}) is a natural generalisation to variable coefficients of the characteristic exponent (\ref{GLK2}) which appears in Gangolli's L\'{e}vy--Khinchine formula (\ref{GLK1}). The eigenrelation (\ref{eig1}) is easily derived using the functional equation for spherical functions
$$\int_{K}\phi(gkh)dk = \phi(g)\phi(h) $$
for all $g, h \in G$, and the fact that $\mu(\sigma, kd\tau) = \mu(\sigma, d\tau)$ for all $k \in K$ within (\ref{PMP6}).

Now let $M$ be a compact symmetric space, so that $G$ is a compact Lie group. Let $\G$ be the unitary dual of $G$, i.e. the set of all equivalence classes (with respect to unitary conjugation) of irreducible unitary representations of $G$. We denote by $V_{\pi}$ the representation space of $\pi$, so $\pi(g)$ is a unitary operator on $V_{\pi}$ for all $g \in G$. Then $V_{\pi}$ is finite--dimensional and we write $d_{\pi}:=$dim$(V_{\pi})$.  We denote by $\G_{S}$ the subset of $\G$ comprising irreducible {\it spherical representations}, so $\pi \in \G_{S}$ if and only if there exists $u \in V_{\pi}$ such that $\pi(k)u = u$ for all $k \in K$. In fact, the subspace of all such $u$ is one--dimensional, and from now on we fix $||u|| = 1$.\footnote{For background on spherical representations, see e.g. \cite{Wol}.}

Every spherical function on $G$ is bounded and positive definite and takes the form $\phi_{\pi}$, where $\pi \in \G_{S}$ and
$$ \phi_{\pi}(g) = \la u, \pi(g)u \ra$$
for all $g \in G$ (see e.g. \cite{Wol}). In particular, it follows that $|\phi_{\pi}(g)| \leq 1$.

 The theory of pseudo--differential operators on compact groups and homogeneous spaces has been developed by Ruzhansky and Turunen in \cite{RT}. As in our previous paper \cite{AT1}, we take a pragmatic approach to this concept and define these operators in what we hope will be a very straightforward manner. By the Peter--Weyl theorem in compact symmetric spaces, we can write $f \in C^{\infty}(K \backslash G /K)$ as a uniformly convergent series

\begin{equation} \label{Fours}
 f = \sum_{\pi \in \G_{S}}d_{\pi}\widehat{f}_{S}(\pi)\phi_{\pi},
 \end{equation}

where $\widehat{f}_{S}(\pi) = \int_{G}f(g)\phi_{\pi}(g)dg$ is the {\it spherical transform} of $f$. We say that a linear operator $T:C^{\infty}(K \backslash G / K) \rightarrow C(K \backslash G/K)$ is a {\it spherical pseudo--differential operator} if there is a mapping $j_{T}:G \times \G_{S} \rightarrow \C$ so that for all $f \in C^{\infty}(K \backslash G /K), \sigma \in G$,

\begin{equation} \label{PSD}
Tf(\sigma) = \sum_{\pi \in \G_{S}}d_{\pi}j_{T}(\sigma, \pi)\widehat{f}_{S}(\pi)\pi(\sigma),
\end{equation} 

We then say that the mapping $j_{T}$ is the {\it spherical symbol} of the operator $T$. Our goal for the remainder of this section is to show that a Gangolli operator of the form (\ref{PMP6}) is a spherical pseudo--differential operator with symbol $j_{T}(\sigma, \pi) = - \eta_{\sigma, \pi}$, where we have written $\eta_{\sigma, \pi}$ instead of $\eta_{\sigma, \phi_{\pi}}$. Our proof is very similar to the corresponding result in \cite{AT1}, where we studied operators on Lie groups having matrix--valued symbols. There are a few places where some additional ideas are needed, so we give a rather concise account, with the emphasis on those points where the proof needs embellishing.

To follow the procedure of section 5 in \cite{AT1}, we will exploit the one--to--one correspondence between $\G$ and the set of dominant weights on $\g$, which is denoted ${\it D}$. In fact we will only need the spherical dominant weights $D_{S}$ which are in one--to--one correspondence with $\G_{S}$. If both $G$ and $K$ are connected, $D_{S}$ is completely described by the Cartan--Helgason theorem (see Theorem 11.4.10 in \cite{Wol}, pp.246--9) but we will not require that result herein. From now on we will write $\pi_{\lambda}$ when $\lambda$ is the weight corresponding to $\pi$, with obvious changes to other notation where indexing will be by weights rather than by representations.

We will equip $\g$ with an Ad--invariant inner product, and write the associated norm as $|\cdot|$. This induces a norm on ${\it D}$ which is denoted by the same symbol. All results that follow in this paragraph are taken from \cite{Sug} (see also Chapter 3 of \cite{AppLbk}). Writing $d_{\lambda}: = d_{\pi_{\lambda}}$, we have the useful estimates
\begin{equation} \label{est1}
d_{\lambda} \leq C_{1} |\lambda|^{M},
\end{equation}
where $C_{1} \geq 0$ and $M$ is the number of positive roots of $G$, and for all $X \in \g$,  there exists $C_{2} \geq 0$ so that
\begin{equation} \label{est2}
||d\pi_{\lambda}(X)||_{HS} \leq C |\lambda|^{\frac{M+2}{2}}|X|,
\end{equation}
where $|| \cdot||_{HS}$ is the Hilbert--Schmidt norm.
We will also need {\it Sugiura's zeta function} $\zeta: \C \rightarrow \R \cup \{\infty\}$, defined by \begin{equation} \label{Sugz}
\zeta(s) = \ds\sum_{\lambda \in {\it D}_{0}}\frac{1}{|\lambda|^{2s}},
\end{equation}
which converges whenever $2\Re(s) > r$, where $r$ is the rank of $G$ and ${\it D}_{0}:= {\it D} \setminus \{0\}$.

\begin{theorem} \label{est3} For all $\lambda \in D_{S}$, there exists $C > 0$ so that
$$ \sup_{g \in G}|\eta_{g, \lambda}| \leq C(1 + |\lambda|^{2} + |\lambda|^{\frac{M+2}{2}})$$
\end{theorem}

\begin{proof} First observe that for all $\lambda \in D_{S}$ (and writing $c_{\lambda}:=c_{\phi_{\lambda}}$),
$$ c_{\lambda} \leq C_{1}(1 + |\lambda|^{2})$$
where $C_{1} > 0$ (see e.g. \cite{AppLbk} p.50).

For the integral term, we first use a Taylor series expansion and the fact that $d\pi(X)u = 0$ for all $X \in \fk$, to deduce that for all $\pi \in \G_{S}, g \in G$, there exists $0 < \theta < 1$ so that

\bean \phi_{\pi}(g) - 1 & = & \la (\pi(g) - I)u, u \ra \\
& = & \left\la \left(\pi\left(\exp{\left(\sum_{i=1}^{d}x_{i}(g)X_{i}\right)}\right) - I\right)u, u \right \ra \\
& = & \left\la \left( \exp{\left(\sum_{i=1}^{d}x_{i}(g)d\pi(X_{i})\right)} - I\right)u, u \right \ra  \\
& = & \left \la \exp{\left(\theta\sum_{i=1}^{d}x_{i}(g)d\pi(X_{i})\right)} \sum_{i=1}^{m}x_{i}(g)d\pi(X_{i})u, u \right \ra, \eean
and so
\bean |\phi_{\pi}(g) - 1| & \leq & \left|\left|\sum_{i=1}^{m}x_{i}(g)d\pi(X_{i})u\right|\right|_{V_{\pi}}\\
& \leq & \left|\left|\sum_{i=1}^{m}x_{i}(g)d\pi(X_{i})u\right|\right|_{HS}\\  \eean
where we can and do compute the Hilbert--Schmidt norm using an orthonormal basis $\{e_{1}, \ldots, e_{d_{\pi}}\}$ for $V_{\pi}$ in which $e_{1} = u$.

Using the conclusions of the last display, the first moment condition and (\ref{est2}), we deduce that there exists $C_{2} \geq 0$ so that
$$ \sup_{g \in G}\left|\int_{U}(\phi_{\lambda}(\tau) - 1)\mu(g, d\tau)\right| \leq C_{2}\max\{|X_{1}|, \ldots, |X_{m}|\}|\lambda|^{\frac{M+2}{2}}.$$

Since $|\phi_{\lambda}(g)| \leq 1$ for all $\lambda \in D_{S}, g \in G$, we easily deduce that

$$ \sup_{g \in G}\left|\int_{U^{c}}(\phi_{\lambda}(\tau) - 1)\mu(g, d\tau)\right| \leq 2\sup_{g \in G}\mu(g, U^{c}) < \infty,$$
where we use the fact that $g \rightarrow \mu(g, U^{c})$ is continuous, this being a consequence of the assumptions we've made to ensure that the range of $A$ comprises continuous functions (see Theorem 3.7 in \cite{AT1}, and the remark that follows its proof). The result follows.
\end{proof}

\begin{prop} \label{absun}
The series $\sum_{\lambda \in D_{S}}d_{\lambda}\eta_{g, \lambda}\widehat{f_{S}}(\lambda)\phi_{\lambda}(g)$ converges absolutely and uniformly (in $g \in G$), for all $f \in C^{\infty}(K \backslash G/K)$.
\end{prop}

\begin{proof} It follows from the properties of Fourier transforms on compact Lie groups that for all $p \in {\mathbb N}$
$$ \lim_{|\lambda| \rightarrow \infty}|\lambda|^{p}|\widehat{f_{S}}(\lambda) = 0,$$
(see e.g. \cite{Sug} or \cite{AppLbk}, p.78). Hence using (\ref{est2}),  and Theorem \ref{est3}, we see that for any $p \in \mathbb{N}$, there exists $\lambda_{0} \in D_{S}\ \setminus\{0\}$ so that there exists $C > 0$ with
$$ \sup_{g \in G}\sum_{|\lambda| > |\lambda_{0}}d_{\lambda}|\eta_{g, \lambda}\widehat{f_{S}}(\lambda)\phi_{\lambda}(g)| \leq C \sum_{|\lambda| > |\lambda_{0}}\frac{|\lambda|^{M}(1 + |\lambda|^{2} + |\lambda|^{\frac{M+2}{2}})}{|\lambda|^{p}}.$$
Now choose $p > \frac{3(M+1)}{2} + r$ and the result follows from (\ref{Sugz}).
\end{proof}

\begin{theorem} $A$ is a pseudo--differential operator of the type (\ref{PSD}) with symbol $j_{A}(g, \lambda) = -\eta_{g, \lambda}$ for all $g \in G, \lambda \in D_{S}$.
\end{theorem}

\begin{proof} To obtain this result, we just imitate the proof of Theorem 5.5 in \cite{AT1}, i.e. utilise (\ref{Fours}) to approximate $f$ by a finite Fourier series $f_{n}$,  and then use the fact that all operators that satisfy the positive maximum principle are closed on their maximal domain (see \cite{EK} Lemma 2.11, p.16) to deduce that for all $g \in G$, $$Af(g) = \limn Af_{n}(g) = -\sum_{\lambda \in D_{S}}d_{\lambda}\eta_{g, \lambda}\widehat{f_{S}}(\lambda)\phi_{\lambda}(g).$$
\end{proof}

Now consider a matrix--valued ``symbol'' (see \cite{AT1}) of ``Gangolli--type'' given by

$$ \sigma(g, \pi) = \left\{ \begin{array}{c c} - a(g)c_{\pi}I_{\pi} + \int_{G}(\pi(\tau) - I_{\pi})\mu(g, d\tau) &~\mbox{if}~\pi \in \G_{S}\\
0 &~\mbox{if}~\pi \notin \G_{S} \end{array}\right., $$
for all $g \in G$.
Then it is easy to see that $\int_{K}\int_{K}\la u, \sigma(kgk^{\prime}, \pi)u\ra dk dk^{\prime}$ is the symbol (taking the form (\ref{eig1})) of a Gangolli operator. This is related to the discussion of ``averaging'' of symbols which can be found on p.671 of \cite{RT}; however in that case, the averaging of the symbol was only with respect to the first variable. We have not shown that $\sigma$ really is the symbol of a pseudo--differential operator; our only goal here is to show how averaging might be implemented.

Finally we remark that we expect that Gangolli operators will also be pseudo--differential operators (with scalar symbols) in the cases where $M$ is non--compact (where the constant coefficient case was discussed in \cite{App5}), and when it is of Euclidean type.

\end{document}